\documentclass[11pt,leqno]{amsart}
\usepackage{amsthm,amsfonts,bm,amssymb,amsmath,color}%
\usepackage{graphicx}
\numberwithin{equation}{section}
\usepackage{tikz}

\usepackage{pgfplots}

\usepackage{mathtools,hyperref}

\renewcommand\d{\partial}
\renewcommand\a{\alpha}

\newcommand\R{\mathbb R}\newcommand\Z{\mathbb Z}
\newcommand\C{\mathbb C}

\def\l{\lambda}

\def\epsilon{\varepsilon}
\def\e{\varepsilon}

\def\dsp{\displaystyle}

\newcommand\br{\begin{rem}}
\newcommand\er{\end{rem}}
\newcommand\bp{\begin{pmatrix}}
\newcommand\ep{\end{pmatrix}}
\newcommand\be{\begin{equation}}
\newcommand\ee{\end{equation}}
\newcommand\ba{\begin{equation}\begin{aligned}}
\newcommand\ea{\end{aligned}\end{equation}}

\newcommand{\T}{{\mathbb T}}

\newcommand{\Id}{{\rm Id }}

\newtheorem{defi}{Definition}[section]
\newtheorem{theo}{Theorem}

\newtheorem{lem}[defi]{Lemma}

\newtheorem{rem}[defi]{Remark}

\numberwithin{equation}{section}

\begin{document}

\title[Dispersive regularization for phase transitions]{Dispersive regularization for phase transitions}

\author{Federico Cacciafesta}
\address{Dipartimento di Matematica, Univesit\`a degli Studi di Padova}, \email{cacciafe@math.unipd.it}
\author{Marta Strani}
\address{Dipartimento di Scienze Molecolari e Nanosistemi, Universit\`a Ca' Foscari Venezia} %
 \email{marta.strani@unive.it}
\author{Benjamin Texier}
\address{Institut Camille Jordan UMR CNRS 5208, Universit\'e Claude Bernard - Lyon 1} \email{texier@math.univ-lyon1.fr} 

\begin{abstract} We introduce a dispersive regularization of the compressible Euler equations in Lagrangian coordinates, in the one-dimensional torus. We assume a Van der Waals pressure law, which presents both hyperbolic and elliptic zones. The dispersive regularization is of Schr\"odinger type. In particular, the regularized system is complex-valued. It has a conservation law, which, for real unknowns, is identical to the energy of the unregularized physical system. The regularized system supports high-frequency solutions, with an existence time or an amplitude which depend strongly on the pressure law.   
\end{abstract}

\maketitle

\section{Introduction}

Given a Van der Waals pressure law, that is such that
\be \label{vdw0}
 p'(u_1) < 0, \qquad \mbox{for some range of $u_1,$}
\ee
 the compressible Euler equations in Lagrangian coordinates 
\be \label{vdw} \left\{\begin{aligned}
\d_t u_1 + \d_x u_2 & = 0 ,\\ \d_t u_2 + \d_x (p(u_1)) & = 0
\end{aligned}\right.\ee
have an elliptic principal symbol. As a consequence, the associated initial-value problems are ill-posed in Sobolev spaces \cite{Met,NT1,NT2}, and even in Gevrey space \cite{Mor1,Mor2,Mor3}. 

In \eqref{vdw}, the real unknown $u_1$ is a fluctuation of density and $u_2$ is a velocity. The anomalous behavior \eqref{vdw0} means that as the density increases, the pressure diminishes. This is typical of a phase transition. We discuss phase transitions and the link between the mathematical models and the physical Van der Waals pressure law in Section \ref{sec:transition}. 

Given the ill-posedness of the initial-value problems for \eqref{vdw0}-\eqref{vdw}, we put forward a regularization of \eqref{vdw0}-\eqref{vdw} which allows for a form of well-posedness. Our regularization is {\it not diffusive}, but {\it dispersive}: we consider the system 
\be \label{vdw:i} \left\{\begin{aligned}
\d_t u_1 + \d_x u_2 & = 0 ,\\ \d_t u_2 + \d_x (p(u_1)) + i \e \d_x^2 u_2 & = 0
\end{aligned}\right.\ee
with $\e > 0,$ $t \geq 0$ and $u_1(t),u_2(t): \, \T \to \C.$ Compared to diffusive or real dispersive regularizations, the upside of \eqref{vdw:i} is that the regularization is energy-preserving in some sense. The downside is that we lose the real character of $u_1$ and $u_2.$

For some mathematical pressure functions $p$ which exhibit the anomalous behavior \eqref{vdw0}, we prove local-in-time well-posedness for the initial-value problems for \eqref{vdw:i} for high-frequency (that is with characteristic frequencies $O(1/\e^2)$), zero-mean and negative-energy data in $H^1(\T) \times L^2(\T).$ For some pressure laws, we have to restrict to small solutions (with respect to $\e),$ for others to small time; see Theorems \ref{th:1} and \ref{th:2}.

We then modify \eqref{vdw:i} and consider 
\be \label{modified}
 \left\{\begin{aligned}
\d_t u_1 + \d_x \overline{u_2} & = 0 ,\\ \d_t u_2 + \d_x \overline{p(u_1)}  + i \e \d_x^2 u_2 & = 0,
\end{aligned}\right.\ee
where $\bar z$ is the complex conjugate of $z \in \C.$ If $\e = 0$ and the pressure law $p$ satisfies \eqref{vdw0}, then the principal symbol for \eqref{modified} is not hyperbolic, hence the initial-value problems for \eqref{modified} are ill-posed in Sobolev and Gevrey spaces. For a pressure law $p$ satisfying \eqref{vdw0}, we prove however that the initial-value problem for \eqref{modified} is well-posed in $H^1(\T) \times L^2(\T),$ for high-frequency, zero-mean and negative energy data. The point here is that the solutions are $O(1)$ and defined in time $O(1)$ with respect to $\e;$ see Theorem \ref{th:3}. 

\subsection{Phase transitions} \label{sec:transition} 
The physical model is \eqref{vdw} at $\e = 0,$ in which $u_1$ represents a fluctuation of density and $u_2$ a velocity. The Van der Waals pressure law typically describes the pressure as a function of specific volume $v$ (equal to the volume typically occupied by a particle of fluid) in real (as opposed to ideal) fluids. For some physical constant $b > 0,$ we have $v > b,$ and the Van der Waals pressure law is 
\be \label{physical:vdw}
p(v) = \frac{k_B T}{v - b} - \frac{a}{v^2},
\ee
where $k_B$ is the Boltzmann constant, $T$ is temperature, and $a$ is another physical constant. For $T < T_c,$ where $k_B T_c = 8 a/(27 b),$ the pressure profile has the form shown on Figure \ref{fig1}. For $\underline v < v < \bar v,$ we have $\d_v p > 0,$ meaning that if we squeeze the gas, the pressure diminishes. For small $v,$ the fluid is in liquid state: a large change of pressure entails only a small change in specific volume; for large $v,$ the fluid is a gas: a small change of pressure implies a large change in volume. The phase transition takes place in the intermediate range $p \in [\underline p, \bar p]:$ given $p$ in this range, there are three possible configurations of the gas. 

A mathematical model for the shape of the pressure in terms of the fluctuation of density $u_1$ is 
\be \label{math:vdw}
p_0(u_1) = - u_1 + u_1^3.
\ee
The mathematical pressure \eqref{math:vdw} is compared with the physical model \eqref{physical:vdw} in Figure \ref{fig1}. The cubic model \eqref{math:vdw} makes sense in view of the profile \eqref{physical:vdw} for $p \in [\underline p, \bar p].$ Note that $u_1$ is related to $1/v,$ the inverse of the specific volume, so that in the mathematical pressure law, the gas state (low density) is to the left, and the liquid state (high density) to the right.

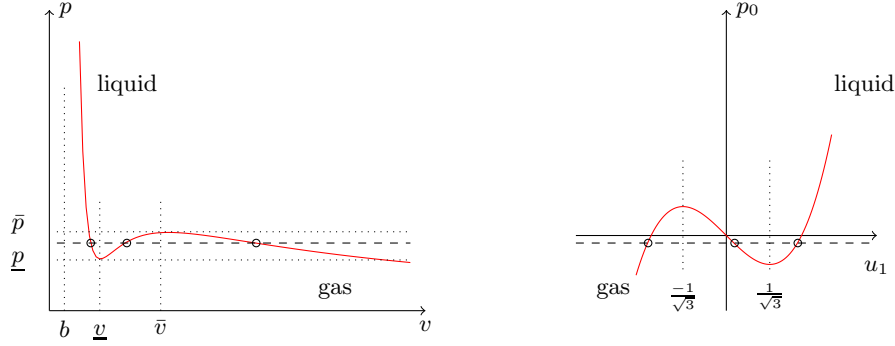
\begin{figure} 
\begin{tikzpicture}

\draw[->] (-6,0) -- (-1,0) ;
\draw[->] (-6, 0) -- (-6, 4); 
\draw (-6,4) node[anchor=west] {\footnotesize $p$} ; 
\draw (-1,0) node[anchor=north] {\footnotesize $v$} ;

\begin{scope}[shift={(3,1)}]
\draw[->] (-2,0) -- (2,0) ; 
\draw[->] (0, -1) -- (0, 3); 
\draw (0,3) node[anchor=west] {\footnotesize $p_0$} ;
\draw (2,-.2) node[anchor=north] {\footnotesize $u_1$} ; 
\draw[domain=-1.2:1.4, samples=100,red]
plot( \x, {  - \x + (\x)^3 } ) ; 
\draw[dashed] (-2,-.1) -- (2,-.1) ; 

\draw (-.577,-.5) node[anchor=north] {\tiny $\frac{-1}{\sqrt 3}$} ;

\draw (.577,-.5) node[anchor=north] {\tiny $\frac{1}{\sqrt 3}$} ; 

\draw (-1.04,-.1) circle (0.05cm) ; 
\draw (.11,-.1) circle (0.05cm) ; 
\draw (.95,-.1) circle (0.05cm) ; 

\draw (-1.5, -.5) node[anchor=north] {\footnotesize gas} ; 
\draw (1.3,2) node[anchor=west] {\footnotesize liquid} ; 

\draw[dotted] (-.577,-.45) -- (-.577,1) ;

 \draw[dotted] (.577,-.45) -- (.577,1) ;
\end{scope}

\begin{scope}[shift={(-6.2,0)}]
\draw[domain=.6:5, samples=100,red]
plot( \x, { 4.05/(\x -.4) - 4*1.5/(\x^2) } ) ;

\draw[dotted] (.4,0) -- (.4, 3) ;

\draw (.4,0) node[anchor=north] {\footnotesize $b$ } ; 

\draw[dotted] (.87,0) -- (.87,1.5) ;

\draw (.87,-.05) node[anchor=north] {\footnotesize $\underline v$ } ; 

\draw[dotted] (1.68,0) -- (1.68,1.5) ; 

\draw (1.68,0) node[anchor=north] {\footnotesize $\bar v$} ;

\draw[dotted] (0.3,1.05) -- (5,1.05) ;

\draw[dotted] (0.3,.675) -- (5,.675) ; 

\draw (0,1.15) node[anchor=east] {\footnotesize $\bar p$} ; 

\draw (0,.655) node[anchor=east] {\footnotesize $\underline p$} ; 

\draw (.7,3) node[anchor=west] {\footnotesize liquid} ; 

\draw (4,.5) node[anchor=north] {\footnotesize gas} ; 

\draw[dashed] (0.3, .9) -- (5,.9) ;

\draw (.75, .9) circle (.05cm) ; 
 
 \draw (1.23, .9) circle (.05cm) ; 
 
 \draw (2.95, .9) circle (.05cm) ; 
 
\end{scope}

\end{tikzpicture}

\caption{To the left, the physical pressure law \eqref{physical:vdw} as a function of the specific volume $v.$ The anomalous behavior corresponds to $\d_v p > 0,$ for $v \in [\underline v, \bar v].$ To the right, the mathematical model \eqref{math:vdw} as a function of the fluctuation of density $u_1.$ The anomalous behavior corresponds to $p_0'(u_1) < 0,$ for $u_1 \in [-1/\sqrt 3,1/\sqrt 3].$} \label{fig1}

\end{figure}

\subsection{Pressure and conservation laws} 

We consider \eqref{vdw:i} with $p \in \{p_0, p_1, p_2\},$ with 
 \be \label{pressure:laws}
 p_0(u) = - u + u^{3}, \qquad p_1(u) =  - u + |u|^{2} u, \qquad p_2(u) = - u + \bar u^{3},
\ee
where $\bar u$ is the complex conjugate of $u.$ For real $u,$ these laws coincide with \eqref{math:vdw}. Their relationship with the physical Van der Waals pressure law \eqref{physical:vdw} is discussed above in Section \ref{sec:transition}.

Associated with $p_1$ and $p_2,$ the regularized system \eqref{vdw:i} has an $\e$-independent conservation law. Indeed, given a smooth solution $(u_1,u_2)$ to \eqref{vdw:i}, we have with the second equation in \eqref{vdw:i}: 
$$ \Re e \, (\d_t u_2, u_2)_{L^2(\T)} -  \Re e \, (p(u_1), \d_x u_2)_{L^2(\T)} - \e \Re e\, (i \d_x^2 u_2, u_2)_{L^2(\T)} = 0.$$ 
The third scalar product above vanishes by symmetry. Thus, using the first equation in \eqref{vdw}, 
\be \label{for:cons:law} \Re e \, (\d_t u_2, u_2)_{L^2(\T)} + \Re e \, (p(u_1), \d_t u_1)_{L^2(\T)} = 0.\ee
With the pressure law $p_1,$ we have
$$ \Re e \, \int_{\T} |u_1|^{2} u_1 \d_t \bar u_1 \, dx = \frac{1}{4} \d_t \int_{\T} |u|^{4} \, dx.$$ 
With the pressure law $p_2,$ we have
$$ \Re e \, \int_{\T} \bar u_1^{3} \d_t \bar u_1 \, dx = \frac{1}{4} \d_t \int_{\T} \Re e \, \bar u^{4}(t,x)\, dx.$$
Thus with 
$$ P(u) = \frac{1}{4} |u|^{4}, \,\, \mbox{if $p = p_1$}, \qquad \mbox{or} \quad P(u) = \frac{1}{4} \Re e \, u^{4}, \,\, \mbox{if $p = p_2.$}$$
 we see that the energy ${\mathcal E}$ is conserved in time for the solutions to \eqref{vdw:i}:
\be \label{cons:law}
 {\mathcal E}(u_1,u_2) = \int_{\T} \big( P(u_1(t,x)) - \frac{1}{2} |u_1(t,x)|^2 +  \frac{1}{2} | u_2(t,x)|^2 \big) \, dx \equiv\mbox{constant.}
\ee
Notice that the energy \eqref{cons:law} does not depend on $\e,$ and, for real $u$, coincides with the energy of the original system \eqref{vdw}.

For the modified system \eqref{modified}, positing a smooth solution $(u_1,u_2),$ we have \eqref{for:cons:law}. If $p = p_0$, this leads to the conservation of a modified energy :
\be \label{modified:cons:law}
 {\mathcal E}(u_1,u_2) = \int_{\T} \Big( \frac{1}{4} \Re e \, u_1(t,x)^4 - \frac{1}{2} \Re e \, u_1(t,x)^2 +  \frac{1}{2} | u_2(t,x)|^2 \,\Big) \, dx \equiv\mbox{constant.}
 \ee

\subsection{First result: $O(1)$ solutions in time $O(\e^2)$ with $p = p_0$}

High-frequency data with zero mean generate solutions of \eqref{vdw:i} which are defined over small time intervals:  

\begin{theo} \label{th:1} For some $T > 0,$ for $\e$ small enough: given the pressure law $p = p_0$ \eqref{pressure:laws}, given data $(u_1^0, u_2^0) \in H^1(\T) \times L^2(\T),$ such that 
$$ %
\int_{\T} u_1^0(x) \, dx = \int_{\T} u_2^0(x) \, dx = 0, \quad \mbox{and} \quad \max(\| u_1^0\|_{H^1},\| u_2^0 \|_{L^2}) < 1/6,
$$ %
 there exists a unique solution $u_\e \in C^0([0, \e^2 T], H^1(\T) \times L^2(\T))$ to the system \eqref{vdw:i} with the datum $$u_\e(0,x) =\left(u_1^0\left(\frac{x}{\e^2}\right), \,  u_2^0\left(\frac{x}{\e^2} \right) \, \right).$$
\end{theo}

The proof is based on a reduction to normal form followed by a standard fixed point theorem. Note the smallness of the time interval. This is remedied in Theorem \ref{th:2}, but only for small-amplitude solutions, in the cases $p = p_1$ and $p = p_2.$

\subsection{Second result: small solutions in time $O(1)$ with $p = p_1$ or $p = p_2$} Small high-frequency data with zero mean and negative energy generate solutions of \eqref{vdw:i} which are defined over time intervals independent of $\e,$ if $ p = p_1$ or $p = p_2$ \eqref{pressure:laws}:

\begin{theo} \label{th:2} For some $T >0,$ for $\e$ small enough: given the pressure law $p = p_1$ or $p = p_2$ \eqref{pressure:laws}, given data $(u_1^0, u_2^0) \in H^1(\T) \times L^2(\T),$ such that 
\be \label{negative:energy}
\int_{\T} u_1^0(x) \, dx = \int_{\T} u_2^0(x) \, dx = 0, \quad \max(\| u_1^0\|_{H^1}, \| u_2^0 \|_{L^2}) < 1/6, \quad \mbox{and} \quad {\mathcal E}(u_1^0, u_2^0) \leq 0, 
\ee
where the energy ${\mathcal E}$ is defined in \eqref{cons:law}, there exists a unique solution $u_\e \in C^0([0, T], H^1(\T) \times L^2(\T))$ to the system \eqref{vdw:i} with the datum $$u_\e(0,x) = \e^{\a} \left(u_1^0\left(\frac{x}{\e^2}\right), \,  u_2^0\left(\frac{x}{\e^2} \right) \, \right).$$
with $\a = 1/2$ if $p = p_1$ and $\a = 1/4$ if $p = p_2.$  
\end{theo}

Compared to Theorem \ref{th:1}, the extra ingredient in the proof of Theorem \ref{th:2} is the exploitation of the conservation of energy in order to extend the solutions up to an observable time interval, and, in the case $p = p_2,$ an integration by parts in time that takes advantage of fast oscillations. See Section \ref{sec:outline} for more details on the proof. For the proof to go through, the solutions however need to have a small amplitude. This is remedied in Theorem \ref{th:3}, but only for the modified system \eqref{modified}.

\subsection{Third result: $O(1)$ solutions in time $O(1)$ for the modified system with $p = p_0$}

For the modified system \eqref{modified} with the pressure law $p = p_0,$ we have large solutions in time $O(1).$ The statement of the theorem involves the norm $c > 0$ of the Sobolev embedding $H^1(\T) \hookrightarrow L^\infty(\T).$ 

\begin{theo} \label{th:3} Given $\l \in (0, 1/(c \sqrt 2)],$ for some $T > 0,$ for $\e$ small enough: given the pressure law $p = p_0$, given data $(u_1^0, u_2^0) \in H^1(\T) \times L^2(\T),$ such that 
\be \label{negative:energy:modified}
\int_{\T} u_1^0(x) \, dx = \int_{\T} u_2^0(x) \, dx = 0, \quad \max(\| u_1^0\|_{H^1},\| u_2^0 \|_{L^2}) < 1/6, \quad \mbox{and} \quad {\mathcal E}(u_1^0, u_2^0) \leq 0, 
\ee
where the energy is defined in \eqref{modified:cons:law}, there exists a unique solution $u_\e \in C^0([0, T], H^1(\T) \times L^2(\T))$ to the system \eqref{modified} with the datum $$u_\e(0,x) = \l \left(u_1^0\left(\frac{x}{\e^2}\right), \,  u_2^0\left(\frac{x}{\e^2} \right) \, \right).$$
\end{theo}

Compared to the proof of Theorem \ref{th:2}, the extra ingredient here is a large, $O(|\ln \e|)$ number of integration by parts in time in the ordinary differential equation in the reduced system. See the discussion in Section \ref{sec:outline} just below.

\subsection{Outline of the proofs} \label{sec:outline}

 The proofs of all three theorems are based on a change of variable to normal form, which uses the Fourier multipler $m$ introduced in Section \ref{sec:m}.  The reduced system consists in a differential equation coupled to a semi-linear Schr\"odinger equation. For the reduced system an elementary application of the Banach fixed point theorem gives a solution in very short time.

Note that we do not need to appeal to Strichartz estimates \cite{Bourgain,BGT} for the Schr\"odinger operator in spite of the low regularity setting. This is due to the fact that the semilinear terms in the Schr\"odinger equation \eqref{eq:w2} all involve the operator $m$ applied to the unknown in that Schr\"odinger equation, where $m$ is the regularizing Fourier multiplier studied in Section \ref{sec:m}, and which in particular maps $L^2(\T)$ into $L^\infty(\T).$

 In the proofs of Theorems \ref{th:2} and \ref{th:3}, in addition to the reduction to normal form, we use the conservation of energy. For data with negative energy, the conservation of energy implies a control of the second unknown in the reduced system in terms of the first one. 
 
 In the proof of Theorem \ref{th:2}, this allows for a continuation of the solutions from a time $o(1)$ to a time $O(1).$ 
 
 In the proof of Theorem \ref{th:3}, we use the abovementioned arguments of reduction to normal form and conservation of energy, and in addition we exploit the fast oscillations in time. A sequence of $O(|\ln \e|)$ integration by parts leads to non-singular (with respect to $\e$) estimates for the ordinary differential equation, in spite of the original large $O(1/\e)$ prefactor in front of the non-linear term. Together with the conservation of energy, this leads to an existence time $O(1)$ for solutions with amplitude $O(1).$ 
 
 The condition $\max(\| u_1^0 \|_{H^1},\| u_2^0 \|_{L^2}) < 1/6$ is naturally not optimal, and the proofs of Theorems \ref{th:1} and \ref{th:2} can be adapted to the case $\max(\| u_1^0 \|_{H^1}, \| u_2^0 \|_{L^2}) < 1.$ In Theorem \ref{th:3}, the constraint $\dsp{\l \leq \frac{1}{c \sqrt 2}}$ is not optimal either. The proof of Theorem \ref{th:3} goes through for any $\l$ and any (zero-mean and negative energy) $(u_1^0, u_2^0)$ with 
 $\dsp{\Big(1 + \frac{(c \l)^2}{(1 - (c \l)^2)^2} \Big) \big( \| u_1^0\|_{H^1} + \|u_2^0\|_{L^2} \big) < 1.}$
Any odd non-linear pressure law works in our framework, in the sense that we could change $p_0$ into $p_0(x) = - u + u^{2n + 1},$ for any $n \geq 1,$ and similarly $p_1(x) = - u + |u|^{2n} u,$ $p_2(x) = - u + \bar u^{2 n + 1},$ and all three theorems would still hold.

\subsection{Background and references} {\it On dispersive regularization:}  
The concept of dispersive regularization, or stabilization, was introduced by G. M\'etivier and J. Rauch in \cite{MR}. In the case of high-frequency data, the results of \cite{MR} were extended first in \cite{TZ}, where a Nash-Moser scheme was used in the existence proof, then by P. Baldi and E. Haus in \cite{BH} and I. Ekeland and E. S\'er\'e in \cite{ES}.

 {\it On ill-posedness of non-hyperbolic initial-value problems:}  Non-hyperbolic initial-value problems are ill-posed in Sobolev spaces: see G. M\'etivier \cite{Met} for first-order fully nonlinear systems. Initial-value problems in which the principal symbol transitions away from hyperbolicity are also ill-posed: see N. Lerner, Y. Morimoto and C.-J. Xu \cite{LMX}, and the extension of their results to systems in \cite{LNT}. The results of \cite{Met} were extended in \cite{NT1,NT2} to systems transitioning away from hyperbolicity; and to Gevrey spaces by B. Morisse in \cite{Mor1,Mor2,Mor3}. Numerical evidence of the instabilites of \cite{LMX}, or lack thereof, was investigated in \cite{ST}.

\section{A Fourier multiplier} \label{sec:m}

We let $\T = \R/(2 \pi \Z),$ and we denote $H^s(\T)$ the standard Sobolev space of $\C$ or $\C^2$-valued maps defined over $\T,$ and $H^s_0(\T)$ the hyperplane of $H^s(\T)$ comprising maps with zero mean. We define the operator $m: L^2(\T) \to H^1_0(\T)$ by 
\be \label{def:m}
 m u := \sum_{ k \in \Z \setminus \{ 0 \} } \frac{e^{ i k x}}{ k} u_k, \qquad u_k := \frac{1}{2 \pi} \int_{\T} e^{- i k y} u(y) \, dy. 
\ee 
 We denote $\Pi_0$ the linear form 
$$ \Pi_0: \qquad u \in L^1(\T) \to u_0 = \int_\T u(x) dx,$$
with $u(x) \in \C$ or $\C^2,$ depending on the context.  

\begin{lem} \label{lem:m} The operator $m$ is linear and bounded from $H^s(\T)$ to $H^{s+1}_0(\T),$ for any $s \in \R,$  with
\be \label{est:m:L2}
 \| m u \|_{H^{s+1}} \leq \| u - \Pi_0 u \|_{H^{s}} \leq \| u \|_{H^s}.
\ee 
Besides,
\be \label{m:1} (m u)(x) = (m u)(0) + i \int_0^x (u(y) - \Pi_0 u) dy.\ee 
 In particular, the operator $m$ is bounded from $L^2(\T)$ to $L^\infty(\T):$ for some $C> 0,$ for all $u \in L^2(\T),$ 
\be \label{est:m:pointwise}
\| m u \|_{L^\infty} \leq C \| u - \Pi_0 u \|_{L^2} \leq C \| u \|_{L^2}, 
\ee 
and  
 \be \label{m:2} -i \d_x m = {\rm Id} - \Pi_0, \qquad - i m \d_x = {\rm Id} - \Pi_0,\ee
 so that
 \be \label{m:3}
  - i \d_x^2 m = \d_x, \quad - i m \d_x^2 = \d_x.
 \ee   
\end{lem}

\begin{proof} The proof is elementary. By density of smooth maps in $H^s(\T),$ for all $s,$ we can work with a smooth map $u,$ for which we have for all $x \in \T$  %
$$ u(x) - \Pi_0 u= \sum_{k \neq 0} e^{i k x} u_k,$$ 
implying
$$ i \int_0^x (u(y) - \Pi_0 u) \, dy = i \sum_{k \neq 0} \frac{1}{ik} (e^{i k x} - 1) u_k,$$ so that 
$$ \begin{aligned} i \int_0^x (u(y) - \Pi_0 u) \, dy & = (m u)(x) - \sum_{k \neq 0} \frac{1}{k} u_k = (m u)(x) - (m u)(0),\end{aligned}$$
and this is \eqref{m:1}.   
Differentiating the above in $x,$ we find $\d_x m = i ({\rm Id} - \Pi_0).$ Besides,
$$ m (\d_x u)(x) = \sum_{k \neq 0} \frac{e^{ik x}}{k} (\d_x u)_k,$$
with $(\d_x u)_k = i k u_k,$ so that
$$ m (\d_x u)(x) =  i \sum_{k \neq 0} e^{i k x} u_k = i (u(x) - \Pi_0 u).$$ 
It then suffices to apply $\d_x$ to the left or to the right in order to find \eqref{m:3}. 
\end{proof}

\section{The initial-value problem for high-frequency data} \label{sec:proof:starts}

The proofs of Theorems \ref{th:1} and \ref{th:2} start here.

\subsection{The regularized systems}  By definition of the pressure laws \eqref{pressure:laws}, we have
$$ \begin{aligned} \d_x p_0(u) & = - \d_x u +3 u^2 \d_x u, \\ 
\d_x p_1(u) & = - \d_x u + 2 |u|^{2} \d_x u + u^2 \d_x \bar u,
\\ \d_x p_2(u) & = - \d_x u + 3 \bar u^{2} \d_x \bar u.\end{aligned}$$
We denote $q_j(u)$ the associated maps:
\be \label{def:q} \begin{aligned} 
q_0(u): & \quad v \in \C \rightarrow 3 u^2 v \in \C, \\ 
q_1(u): & \quad v  \in \C \rightarrow 2 |u|^{2} v + u^2 \bar v \in \C, \\
q_2(u): & \quad v \in \C \rightarrow 3 \bar u^{2} \bar v \in \C, 
\end{aligned}
\ee
so that
$$ %
\d_x (p(u)) = - \d_x u + q(u) \d_x u, \quad \mbox{with $(p,q) = (p_j,q_j),$ $j \in \{0,1,2\}.$}
$$ %
Note that in view of \eqref{def:q}, the operator $q_0(u)$ is a multiplication operator in $\C,$ but $q_1(u)$ and $q_2(u)$ involve the conjugation operator $z \to \bar z.$ 
The Euler-Van-der-Waals system \eqref{vdw} takes the form 
$$ %
\d_t u + A \d_x u + q(u_1) B \d_x u = 0, \qquad x \in \T,
$$ %
with  
$$ %
 A := \left(\begin{array}{cc} 0 & 1 \\ - 1 & 0 \end{array}\right),  \qquad B := \left(\begin{array}{cc} 0 & 0 \\ 1 & 0 \end{array}\right).
 $$ %
Crucially, $A$ is not symmetric. The quasi-linear term $q(u_1) B \d_x$ has a nilpotent structure and depends only on the first component of $U.$ 
For $\e > 0,$ and $U  = (u_1,u_2)\in \C^2,$ the regularized system \eqref{vdw:i} is 
\be \label{1}
\d_t u + A \d_x u + q(u_1) B \d_x u + \e i D \d_x^2 u = 0, \qquad D := \left(\begin{array}{cc} 0 & 0 \\ 0 & 1 \end{array}\right).
\ee

\subsection{Small high-frequency data} 

We consider concentrating data with amplitude $\e^\a,$ with $\a \geq 0:$ 
$$ %
u(0,x) = \e^\a u^0(x/\e^2),
$$ %
where $u^0$ is independent of $\e$ and belongs to $H^1(\T) \times L^2(\T).$ In Theorem \ref{th:1}, we have $\a = 0;$ in Theorem \ref{th:2} we have $\a = 1/2$ or $\a = 1/4,$ depending on the pressure law. %
We look for $u$ in the form 
$$ %
u(t,x) = \e^\a \tilde u(t,x/\e^2) = (\e^\a \tilde u_1, \e^\a \tilde u_2)(t,x/\e^2).
$$ %
Thus the initial-value problem takes the form 
\be \label{4} \left\{ \begin{aligned} 
\d_t \tilde u + \frac{1}{\e^2} (A  \d_x +  q(\e^\a \tilde u_1) B \d_x) \tilde u + \frac{i}{\e^3} D \d_x^2 \tilde u & = 0, \qquad t \geq 0, \\ \tilde u(0,x) & =  u^0(x),
\end{aligned}\right.\ee
with $x \in \T.$ 
The original Euler-Van der Waals system is a system of conservation laws. In particular, the mean mode of the solution to Euler-Van der Waals system ($\e = 0$) is conserved over time. Since $\Pi_0 \d_x^2 \equiv0,$ this conservation is true as well for the solution to \eqref{4}: 
\be \label{conserved}
 \Pi_0 \tilde u(t) \equiv\Pi_0 u^0,
\ee
so long as $\tilde u$ is defined. In the statements of Theorems  \ref{th:1} and \ref{th:2}, we assume that the datum has zero mean. Then the solution $\tilde u$ has zero mean so long as it is defined.

\section{Normal form reduction}  \label{sec:normal} 

We now let
\be \label{5}
 v = (\Id + \e M)^{-1} \tilde u, \qquad \mbox{with $M  = - m J,$ \quad $\dsp{J := \left(\begin{array}{cc} 0 & 1 \\ 1 & 0 \end{array}\right),}$}
\ee
with $m$ defined in \eqref{def:m}. In view of Lemma \ref{lem:m}, we have the bounds, for $0 < \e < 1:$ 
\be \label{bd:M:L2}
(1 - \e) \| v \|_{L^2} \leq  \| \tilde u \|_{L^2} \leq (1 + \e) \| v \|_{L^2},
\ee
and 
$$ %
(1 -\e C ) \| v \|_{L^\infty} \leq \| \tilde u \|_{L^\infty} \leq (1 + \e C ) \| v \|_{L^\infty},
$$ %
with the same positive constant $C$ as in the pointwise estimate \eqref{est:m:pointwise} in Lemma \ref{lem:m}. Since $m$ is a Fourier multiplier, it commutes with spatial derivatives, so that estimate \eqref{bd:M:L2} extends to Sobolev spaces:
\be \label{bd:M:Sobolev}
(1 - \e) \| v \|_{H^{s'}} \leq  \| \tilde u \|_{H^{s'}} \leq (1 + \e) \| v \|_{H^{s'}}, \qquad \mbox{for any $s'\in \R.$} 
\ee
We denote $(v_1, v_2)$ the components of $v.$ By conservation of the mean mode \eqref{conserved}, and assumption on the datum, we have 
$$ %
 \Pi_0 (\Id + \e M ) v \equiv0.
$$ %
Since $\Pi_0 m = 0$ (see Lemma \ref{lem:m}), this implies 
\be \label{mmc}
 \Pi_0 v_1  \equiv0, \qquad \Pi_0 v_2 \equiv0,
\ee
so long as $v_j$ are defined.

\subsection{Cancellation and the reduced system} \label{sec:normal:form} The system \eqref{4} in $\tilde u$ and the definition \eqref{5} of $v$ lead to 
\be \label{6} \begin{aligned} 
\d_t v & + \frac{i}{\e^3} D \d_x^2 v + \frac{1}{\e^2} (A  \d_x + [i D \d_x^2, M] + q(\e^\a \tilde u_1) B \d_x) v  \\ & + \frac{1}{\e^2} \big( (\Id + \e M)^{-1} - 1 \big)  \big[ i D \d_x^2, M \big] v  + \frac{1}{\e} (\Id + \e M)^{-1} \big[ A \d_x + q(\e^\a \tilde u_1) B \d_x, M \big] v = 0.
\end{aligned} \ee
By \eqref{m:3} in Lemma \ref{lem:m}, with the definition of $M$ in \eqref{5}, we observe the key cancellation 
$$ %
[i D \d_x^2, M] + A \d_x = - i D J \d_x^2 m + i J D m \d_x^2 + A \d_x = - [J,D] \d_x +  A \d_x = 0.
$$ %
The equation \eqref{6} becomes
\be \label{7} \begin{aligned}
\d_t v & + \frac{i}{\e^3} D \d_x^2 v + \frac{1}{\e^2} q(\e^\a \tilde u_1) B \d_x v  - \frac{1}{\e^2} \big( (\Id + \e M)^{-1} - 1 \big) A \d_x v\\ & + \frac{1}{\e} (\Id + \e M)^{-1} \big[ A \d_x + q(\e^\a \tilde u_1) B \d_x, M \big] v = 0.
\end{aligned} \ee   
By expanding in Neumann series, we observe that 
\be \label{neumann}
 (\Id + \e M)^{-1} - \Id = - \e (\Id  + \e M)^{-1} M,
\ee
so that \eqref{7} becomes
\be \label{7.1} \begin{aligned}
\d_t v + \frac{i}{\e^3} D \d_x^2 v & +  \frac{1}{\e^2} q(\e^\a \tilde u_1) B \d_x v \\ & + \frac{1}{\e} (\Id + \e M)^{-1} \Big( M A \d_x + [A \d_x + q(\e^\a \tilde u_1) B \d_x, M] \Big)  v  = 0.
\end{aligned} \ee 
The $O(\e^{-1})$ operator in \eqref{7.1} is 
\be \label{for:E}  M A \d_x + [A \d_x + q(\e^\a \tilde u_1) B \d_x, M]  = A \d_x M + [ q(\e^\a \tilde u_1) B \d_x, M].\ee
We observe that, by Lemma \ref{lem:m},  
\be \label{10} 
A \d_x M = -i A J (\Id - \Pi_0), \qquad A J  = \left(\begin{array}{cc} 1 & 0 \\ 0 & -1 \end{array}\right).
\ee
Using \eqref{m:2}, we see that the other term in \eqref{for:E} is  
$$ %
[ q(\e^\a \tilde u_1) B \d_x, M ] = -  i q(\e^\a \tilde u_1) B J ({\Id} - \Pi_0) + m q(\e^\a \tilde u_1) \d_x J B.
$$ %
We compute
$$ B J = \left(\begin{array}{cc} 0 & 0 \\ 0 & 1 \end{array}\right), \qquad J B = \left(\begin{array}{cc} 1 & 0 \\ 0 & 0 \end{array}\right),$$
so that
$$ %
[ q(\e^\a \tilde u_1) B \d_x, M ] = \left(\begin{array}{cc} m \circ q(\e^\a \tilde u_1) \d_x & 0 \\ 0 & - i q(\e^\a \tilde u_1) (\Id - \Pi_0) \end{array}\right).
$$ %
We introduce the remainder $r_1$ such that  
\be \label{def:r1} m (q(\e^\a \tilde u_1) \d_x v_1) = m( q(\e^\a v_1) \d_x v_1) + \e r_1,\ee
and, with \eqref{pressure:laws}, we note that  
$$ q(\e^\a v_1) \d_x v_1 = \e^{2 \a} \d_x (p(v_1) + v_1).$$ 
Thus
$$  m( q(\e^\a v_1) \d_x v_1) =  \e^{2 \a} m \d_x (p(v_1) +v_1) = i \e^{2 \a} (\Id - \Pi_0)(p(v_1) + v_1).$$ %
Together with \eqref{10} and \eqref{mmc}, this gives a description of the $O(\e^{-1})$ operator in \eqref{7.1}: we let  
\be \label{def:E}
E v := \left(\begin{array}{c} - i v_1 + i \e^{2 \a} (\Id - \Pi_0) (p(v_1) + v_1) \\ i (1 - q(\e^\a \tilde u_1))  v_2 \end{array}\right).
\ee
The reduced system is 
\be \label{reduced:system}
 \d_t v + \frac{i}{\e^3} D \d_x^2 v  +  \frac{1}{\e^2} q(\e^\a \tilde u_1) B \d_x v + \frac{1}{\e} E  v + R v = 0,
 \ee
 with $E$ defined in \eqref{def:E}, and the remainder $R$ is    
 \be \label{def:R}
 R v := \tilde R v + \left(\begin{array}{c} r_1 \\ 0 \end{array}\right),
 \ee
 where $r_1$ is defined in \eqref{def:r1}, and, using \eqref{neumann} again,
 \be \label{id:R}  \begin{aligned}
 \tilde R := - (\Id + \e M)^{-1} M  E.
\end{aligned}\ee 
 The system \eqref{reduced:system} is partially decoupled. The equation in $v_1$ is a non-linear ordinary differential equation: 
\be \label{eq:checku1}
 \d_t v_1 - i \e^{-1} v_1 + i \e^{2 \a - 1} (\Id - \Pi_0) (p(v_1) + v_1)  = - (R v)_1 ,
 \ee
and the equation in $v_2$ is a semi-linear Schr\"odinger equation: 
\be \label{eq:checku2} \begin{aligned} 
 \d_t v_2 + \frac{i}{\e^3} \d_x^2 v_2 + \frac{i}{\e} (1 - q(\e^\a \tilde u_1)) v_2 & = - \frac{1}{\e^2} q(\e^\a \tilde u_1) \d_x v_1 - (R v)_2.
 \end{aligned}
 \ee 
In \eqref{eq:checku1}-\eqref{eq:checku2} we took into account the fact that $\Pi_0 v_j \equiv0,$ as described in \eqref{mmc}. 

The components $(R v_1)$ and $(R v)_2$ of the remainder $R v$ that appear in \eqref{eq:checku1} and \eqref{eq:checku2} are defined in \eqref{def:R}, in terms of $r_1$ defined in \eqref{def:r1} and $\tilde R$ defined in \eqref{id:R}. We have
$$ %
 (R v)_1  = -  ((\Id + \e M)^{-1} M E v)_1 + r_1, \qquad 
 (R v)_2  = - ((\Id + \e M)^{-1} M E v)_2,   
$$ %

\subsection{Time oscillations} \label{sec:time}

 At this point we factorize the fast time oscillations as we let
\be \label{def:tildeu} 
  w_1(t,x) := e^{- i t/\e} v_1(t,x), \qquad 
   w_2(t,x)  := e^{i t/\e} v_2(t,x).
\ee

From \eqref{eq:checku1} and \eqref{def:tildeu}, we deduce the equation in $w_1.$ We have
\be \label{eq:w1:0}
\d_t w_1 = - i \e^{2 \a - 1} e^{2 i t/\e} (\Id - \Pi_0)(w_1^3) - R_1, \qquad \mbox{if $p = p_0;$}
 \ee
\be \label{eq:w1:1}
 \d_t w_1 = - i \e^{2 \a - 1} (\Id - \Pi_0)( |w_1|^{2} w_1) - R_1, \qquad \mbox{if $p = p_1;$}
 \ee
 and
 \be \label{eq:w1:2}
 \d_t w_1 = - i \e^{2 \a - 1} e^{ - 4 i t/\e} (\Id - \Pi_0) \bar w_1^{3} - R_1, \qquad \mbox{if $p = p_2,$}
 \ee
with
\be \label{def:R1}
 R_1 := e^{- i t/\e}  (R v)_1,
\ee
in all three cases. Note the fast time oscillations in factor of the nonlinear term in \eqref{eq:w1:2}, a key difference with \eqref{eq:w1:1}.

From \eqref{eq:checku2} and \eqref{def:tildeu}, we deduce the equation in $w_2:$
\be \label{eq:w2} \begin{aligned} 
 \d_t w_2 + \frac{i}{\e^3} \d_x^2 w_2  & = - \frac{1}{\e^2} e^{i t/\e} q(\e^\a \tilde u_1) \big( e^{it /\e} \d_x w_1 \big) + \frac{i}{\e} e^{it /\e}  q(\e^\a \tilde u_1) \big( e^{- i t /\e} w_2 \big) -  R_2,
\end{aligned}
\ee
with $R_2 := e^{i t/\e} ( R v)_2.$
In the rest of the proofs, we work with \eqref{eq:w1:0}-\eqref{eq:w1:1}-\eqref{eq:w1:2}-\eqref{eq:w2}: an ordinary differential equation in $w_1$ coupled to a semilinear Schr\"odinger equation in $w_2.$

\section{End of proof of Theorem \ref{th:1}}  \label{sec:end:1} 

Here the pressure law is $p = p_0$ \eqref{pressure:laws}. We start from \eqref{eq:w1:0}-\eqref{eq:w2}, with $\a = 0.$ We use a simple fixed point argument. Consider the balls 
\be \label{def:balls} B_s(\rho,T) = \{ u \in C^0([ 0, T], H^s(\T)), \quad \| u\|_{L^\infty H^s} \leq \rho \}, \qquad 0 < \rho < 1, \quad s \in \R.\ee
 Product spaces are equipped with the $L^\infty$ norm; in particular we have $\| (f,g)\|_{H^1 \times L^2} = \max( \|f \|_{H^1}, \|g \|_{L^2}).$ Given $(w_1^0, w_2^0) \in B_1(\rho/6,T) \times B_0(\rho/6, T),$ we let
 \be \label{def:F}
 F: B_1(\rho,T) \times B_0(\rho,T) \to C^0([0,T], H^1(\T) \times L^2(\T)),
 \ee
 be defined by $F = (F_1,F_2),$ with  
 \be \label{def:F1:1} 
F_1(w_1,w_2)(t) := w_1^0 - i \e^{- 1} \int_0^t e^{2 i t'/\e} (\Id - \Pi_0)(w_1^3)(t') \, dt'- \int_0^t R_1(t') \, dt', %
\ee
and
\be \label{def:F2} \begin{aligned}
 F_2(w_1,w_2)(t)  & := e^{- i t \d_x^2/\e^3} w_2^0 - \int_0^t e^{ - i (t - t') \d_x^2/\e^3} \big( f_2(t',w_1(t'), w_2(t')) +  R_2(t')\big) \, dt',
\end{aligned}
\ee 
with notation
$$ f_2(t,w_1(t), w_2(t)) := \e^{-2} e^{2 i t'/\e} q_0(\tilde u_1) \d_x w_1(t)  - \e^{-1} i  q_0(\tilde u_1) w_2(t),$$
where $q_0$ is defined in \eqref{def:q}. 
A fixed point of $F$ is a solution to \eqref{eq:w1:0}-\eqref{eq:w2} issued from $(w_1^0, w_2^0)$ in time $[0,T].$ Note that the remainders $R_1$ and $R_2$ which intervene in \eqref{def:F1:1} and \eqref{def:F2} are defined in terms of $\tilde u_1,$ which, via \eqref{5} and \eqref{def:tildeu}, is expressed in terms of $w_1$ and $w_2.$

\subsection{Remainder bounds} \label{sec:remainder}

We derive bounds for $R_1$ defined in \eqref{def:R1} in term of $\tilde R$ and $r_1.$
 The remainder $\tilde R$ is described in \eqref{id:R}, and we can use Lemma \ref{lem:m} for $m.$   
By \eqref{est:m:L2} and \eqref{bd:M:Sobolev}: 
$$ \| \tilde R v \|_{H^s} \lesssim \| M E v \|_{H^s} \leq \| E v \|_{H^{s-1}}.$$ 
The operator $E$ is defined in \eqref{def:E}. Using \eqref{def:q}, we see that it satisfies 
$$ \begin{aligned} \| (E v)_1 \|_{L^2}  \lesssim \big(1 + \|v_1\|_{L^\infty}^{2} \big) \| v_1 \|_{L^2}, \qquad 
 \| (E v)_2 \|_{L^2}  \lesssim \big(1 + \| \tilde u_1 \|_{L^\infty}^{2}\big) \| v_2 \|_{L^2}.
 \end{aligned}$$
By \eqref{est:m:pointwise} and \eqref{5}, we have
\be \label{bd:tildeu1:Linfty} \|  \tilde u_1 \|_{L^\infty} \lesssim \| v_1 \|_{L^\infty} + \e \| v_2 \|_{L^2}.\ee
We turn to $r_1$ defined in \eqref{def:r1}. In view of \eqref{est:m:L2}, we have 
$$ \begin{aligned} \| r_1 \|_{H^1} = \e^{-1} \| m \Big( \, \big( q_0(\tilde u_1) - q_0( v_1) \, \big) \d_x v_1 \, \Big) \|_{H^1}  \lesssim \e^{-1} \| \big( q_0( \tilde u_1) - q_0( v_1)\big) \d_x v_1  \|_{L^2}. \end{aligned}$$ 
By \eqref{def:q}, 
$$  \| \big( q_0( \tilde u_1) - q_0( v_1) \big) \d_x v_1  \|_{L^2} \lesssim C_{q_0}(\| \tilde u_1 \|_{L^\infty} + \| v_1 \|_{L^\infty}) \| \tilde u_1 - v_1 \|_{L^\infty} \| \d_x v_1 \|_{L^2},$$
for some nonegative and nondecreasing function $C_{q_0}.$ Since $\tilde u_1 = v_1 - \e m v_2,$ with \eqref{est:m:pointwise} this gives 
$$ %
 \| r_1 \|_{H^1} \lesssim C_{q_0}(\| v_1 \|_{L^\infty} + \e \| v_2 \|_{L^2}) \| v_2 \|_{L^2} \| \d_x v_1 \|_{L^2}.
$$ %
 Summing up, we obtain 
$$ %
 \| R_1 \|_{H^1} \lesssim \| v\|_{L^2} + C_{q_0}\big(\| v_1 \|_{L^\infty} + \e \| v_2 \|_{L^2}\big) \big( \| v \|_{L^2} + \| v_2 \|_{L^2} \| \d_x v_1 \|_{L^2}\big), 
$$ %
 with another nonegative and nondecreasing function $C_{q_0}$ which depends on $q_0.$ The remainder $R_2$ defined in \eqref{eq:w2} satisfies the same bound in $L^2$ norm. 
 For $(w_1, w_2) \in B_1(\rho, T) \times B_0(\rho, T)$ with $\rho < 1,$ 
 this implies for $\e < 1$ the bound 
 \be \label{remainder:bound}
 \| R_1 \|_{H^1} + \| R_2 \|_{L^2} \lesssim C_R \rho,
 \ee 
 uniformly in $t \in [0,T],$ for some constant $C_R > 0$ which is uniform in $\rho < 1.$ %

\subsection{Bounds on $F$} As $p = p_0,$ $F_1$ is defined by \eqref{def:F1:1}. We note that, for $w_1 \in B_1(\rho, T),$
$$ \|  w_1^3 \|_{H^1} \leq 3 \| w_1 \|_{L^\infty}^{2} \| w_1 \|_{H^1}.$$
Thus, given $(w_1, w_2) \in B_1(\rho, T) \times B_0(\rho, T)$ with $\rho < 1,$ 
$$ \| F_1 \|_{H^1} \leq \| w_1^0\|_{H^1} +  3 \e^{ - 1}  T \| w_1 \|_{L^\infty([0, T], L^\infty)}^{2} \| w_1 \|_{L^\infty([0, T], H^1)} + T \| R_1 \|_{L^\infty([0, T], H^1)}.$$
With the remainder bound \eqref{remainder:bound}, this gives 
\be \label{bd:F1:1} \| F_1 \|_{H^1} \leq \| w_1^0\|_{H^1} + 3 \e^{- 1} c^{2} \rho^{3} T + C_R \rho T,\ee
where $c > 0$ is a norm for the Sobolev embedding $H^1(\T) \hookrightarrow L^\infty(\T),$ and $C_R > 0$ is the constant that appears in the remainder bound \eqref{remainder:bound}.
We turn to $F_2:$ by definition of $q_0$ \eqref{def:q}, we have
$$ \begin{aligned} \| f_2 \|_{L^2} & \leq 3 \| \tilde u_1 \|_{L^\infty}^{2} \big( \e^{-2}  \| \d_x v_1 \|_{L^2} + \e^{-1} \| v_2 \|_{L^2}\big). 
\end{aligned} $$    
With \eqref{bd:tildeu1:Linfty}, for $(w_1, w_2) \in B_1(\rho,T) \times B_0(\rho,T),$ this implies
$$ \| f_2 \|_{L^2} \leq 3 (c + \e)^{2} \rho^{2} (\e^{-2} + \e^{-1}) \rho.$$ 
The Schr\"odinger solution operator $e^{i s \d_x^2}$ being unitary in $L^2(\T)$ for $s \in \R,$ we have  
$$ \| F_2(t,w_1(t), w_2(t)) \|_{L^2} \leq \| w_2^0\|_{L^2} + \int_0^t \big( \| f_2(t', w_1(t'), w_2(t')) \|_{L^2} + \| R_2(t') \|_{L^2} \big) \, dt'.$$ 
With the above bound for $f_2,$ and the remainder bound \eqref{remainder:bound}, this gives for $\e < \min(1,c)$ the bound 
\be \label{bd:F2} \| F_2 \|_{L^2} \leq \| w_2^0 \|_{L^2} + \e^{-2} C \rho^{3} T + C_R \rho T,
\ee
for some constant $C > 0$ which depends only on $c.$ 
 Thus for $T = T_\e$ small enough, depending on $\e$ and $\rho,$ and $\e$ small enough, recalling that $\| w_1^0 \|_{H^1} \leq \rho/6$ and $\| w_2^0 \|_{L^2} \leq \rho/6,$ we have
\be \label{bd:F} \| F(t,w_1(t), w_2(t)) \|_{H^1 \times L^2} \leq \rho, \quad \mbox{for $0 \leq t \leq T_\e,$}\ee
meaning that $F$ maps $B_1(\rho,T_\e) \times B_2(\rho,T_\e)$ to itself. 
How small does $T_\e$ need to be? We see on \eqref{bd:F1:1} and \eqref{bd:F2} that the worst term in the upper bounds is the term $\e^{-2} C \rho^{2}$ in the upper bound for $F_2.$ All the other terms are bounded in $\e$ as $\e \to 0.$ So if we let 
$$ %
 T_\e := \frac{\e^{2}}{2 C \rho^{2}},
$$ %
 where $C$ is the constant which appears in \eqref{bd:F2}, then for $\e$ small enough (depending on $\rho$), we have \eqref{bd:F}. 

\subsection{Contraction bounds} \label{sec:contraction} Given two pairs $(w_1, w_2)$ and $(w'_1, w'_2)$ in $B_1(\rho,T_\e) \times B_0(\rho, T_\e),$ the goal is to bound 
$$ F - F' :=  F(t,w_1(t), w_2(t)) - F(t, w_1'(t), w_2'(t))$$
in $H^1 \times L^2$ norm, uniformly in $t \in [0, T_\e].$ We have
$$ \|  w_1^3 - ( w'_1)^3 \|_{H^1} \leq C_{\rho} \| w_1 - w'_1 \|_{H^1}.$$
for some $C_\rho > 0.$ 
The difference between the remainders $\tilde R$ (defined in \eqref{id:R}) and $\tilde R'$ (same as $\tilde R$ but evaluated at $\tilde u'_1$) satisfies  
$$ \| \tilde R v - \tilde R' v' \|_{H^1} \lesssim \| E v - E' v' \|_{L^2}.$$
Besides,
$$ \| (E v)_1 - (E' v')_1 \|_{L^2} \leq (1  +  C_{\rho}) \| v_1 - v'_1 \|_{L^2},$$
and
$$ \| (E v)_2 - (E' v')_2 \|_{L^2} \leq (1  + C_{\rho}) \| v_2 - v'_2 \|_{L^2} +  C_{\rho} \| v_1 - v'_1 \|_{L^\infty}.$$
Similarly,
$$ \| r_1 - r'_1 \|_{H^1} \leq C_{\rho} (\| v_2 - v'_2 \|_{L^2} + \| \d_x v_1 - \d_x v'_1 \|_{L^2}).$$
Thus from \eqref{def:F1:1} we deduce
$$ \| F_1 - F'_1 \|_{H^1} \leq T_\e C_{\rho} \max\big( \| w_1 - w'_1 \|_{L^\infty([0, T_\e], H^1)}, \| w_2 - w'_2 \|_{L^\infty([0, T_\e], L^2)}\big).$$
The bound for $F_2 - F'_2$ is similar. From there, we deduce that if $\e$ is small enough, depending on $\rho,$ then $F$ is a contraction in $B_1(\rho, T_\e) \times B_0(\rho, T_\e).$ The Banach fixed point theorem then asserts the existence of a unique fixed point for $F$ in $B_1(\rho,T_\e) \times B_0(\rho,T_\e).$ This fixed point is a solution to the initial-value problem for $(w_1, w_2)$ with $(w_1(0), w_2(0)) = (w_1^0, w_2^0).$ A standard continuation argument extends the uniqueness  from the product of balls into the whole space $C^0([0, T_\e], H^1 \times L^2).$

\section{End of proof of Theorem \ref{th:2}}

Here $p = p_1$ and $\a = 1/2,$ or $p = p_2$ and $\a = 1/4.$ We work with \eqref{eq:w1:1}-\eqref{eq:w2} if $p = p_1,$ and with \eqref{eq:w1:2}-\eqref{eq:w2} if $p = p_2.$

\subsection{The conservation law in the rescaled spatial frame} 

For $1 \leq p < \infty,$ given $f \in L^p(\T),$ and $g(x) = \e^\a f(x/\e^2),$ we have 
$$ \| g \|^p_{L^p(\T)} =  \e^{\a p + 2}  \| f \|^p_{L^p}.$$
In particular, for the putative solution $(\tilde u_1,\tilde u_2)$ of \eqref{4}, the energy is 
 \be \label{energy:1} {\mathcal E}_1(t) := \frac{\e^{4 \a + 2}}{4} \int_\T |u_1|^{4}  \, dx  -  \frac{\e^{2\a + 2}}{2} \| \tilde u_1 \|^2_{L^2} + \frac{\e^{2\a + 2}}{2} \| \tilde u_2 \|_{L^2}^2 \equiv {\mathcal E}_1(0) \leq 0,\ee
by assumption \eqref{negative:energy}, if $p = p_1$ \eqref{pressure:laws}, implying
$$ %
 \| \tilde u_2(t) \|_{L^2} \leq  \| \tilde u_1(t)\|_{L^2}, \qquad \mbox{with $p = p_1.$}
$$ %
 If $p = p_2$ \eqref{pressure:laws}, we find
 $$ 
\label{energy:2} {\mathcal E}_2(t) := \frac{\e^{4 \a + 2}}{4} \int_\T \Re e \, u_1^{4}  \, dx  -  \frac{\e^{2\a + 2}}{2} \| \tilde u_1 \|^2_{L^2} + \frac{\e^{2\a + 2}}{2} \| \tilde u_2 \|_{L^2}^2 \equiv {\mathcal E}_2(0) \leq 0.
$$ 
This implies 
\be \label{v2leqv1:p2}
 \| \tilde u_2(t) \|_{L^2}^2 \leq \| \tilde u_1(t)\|_{L^2}^2 - \frac{\e^{2 \a}}{4} \int_\T \Re e \, \tilde u_1^{4}(t,x)  \, dx, \qquad \mbox{with $p = p_2.$}
\ee

\subsection{Conservation of the mean mode and energy} \label{sec:cons:check}

In the case $p = p_1,$ from \eqref{energy:1} and the definition of $v$ in \eqref{5}, we deduce 
$$ %
\| v_2 - \e m v_1 \|_{L^2} \leq \| v_1 - \e m v_2 \|_{L^2}. 
$$ %
This implies
  \be \label{w2leqw1}
  \| v_2 \|_{L^2} \leq \frac{1 + \e}{1 - \e} \| v_1 \|_{L^2}, \qquad \mbox{if $p = p_1.$} 
 \ee
 so long as the solution $(v_1, v_2)$ to \eqref{eq:checku1}-\eqref{eq:checku2} is defined. 

In the case $p = p_2,$ using \eqref{est:m:pointwise} we find, starting from \eqref{v2leqv1:p2}, the bound 
$$ \| v_2 - \e m v_1 \|_{L^2}^2 \leq \big(1 + \e^{2 \a} \big(\| v_1 \|_{L^\infty} + \e \| v_2 \|_{L^2}\big)^{2} \, \big) \big(\| v_1 \|_{L^2}^2 + 2 \e \| v_1 \|_{L^2} \| v_2 \|_{L^2} + \e^2 \| v_2 \|_{L^2}^2 \big). $$ 
This implies
\be \label{w2leqw1:2}
\| v_2 \|_{L^2}^2 \leq \frac{ 1 + \e^{2\a} \big(\| v_1 \|_{L^\infty} + \e \| v_2 \|_{L^2}\big)^{2} }{1 - 4 \e  \big(1 + \e^{2\a}\big(\| v_1 \|_{L^\infty} + \e \| v_2 \|_{L^2}\big)^{2} \big) }  \cdot  (1 + 5 \e) \| v_1 \|_{L^2}^2,
\ee
for $\e < 1.$

\subsection{An integration by parts in time} \label{IBP}
If $p = p_2,$ then based on \eqref{eq:w1:2} we have the implicit representation
\be \label{w1:implicit}
 w_1(t) = w_1(0) - i \e^{2 \a - 1} \int_0^t e^{ - 4 i  t'/\e} (\Id - \Pi_0) \bar w_1^{3}(t') \, dt' - \int_0^t R_1(t') \, dt'.
 \ee
 The fast time oscillations in the nonlinear term allows for an integration by parts and the gain of $\e^{2  n\a}.$ Indeed, we have 
$$ \begin{aligned}
\int_0^t e^{ - 4 i  t'/\e} & (\Id - \Pi_0) \bar w_1^{3}(t') \, dt'
\\ & = \frac{i \e}{4} \Big( e^{ - 4 i  t/\e}  (\Id - \Pi_0) \bar w_1^{3}(t) - (\Id - \Pi_0) \bar w_1^{3}(0)\Big) \\ & -  i \e \int_0^t  e^{ - 4 i t'/\e} (\Id - \Pi_0) \bar w_1^{2}(t') \d_t \bar w_1(t') \, dt'.
\end{aligned}$$
Using \eqref{eq:w1:2} again, and \eqref{w1:implicit}, this gives
\be \label{w1:implicit:2}
\begin{aligned}
w_1(t) & = w_1(0) + \frac{\e^{2 \a}}{4} \Big( e^{ - 4 i  t/\e}  (\Id - \Pi_0) \bar w_1^{3}(t) - (\Id - \Pi_0) \bar w_1^{3}(0)\Big)  \\ & - i \e^{4 \a - 1} \int_0^t (\Id - \Pi_0) \big( \bar w_1^{2} (t') (\Id - \Pi_0) w_1^{3} (t') \big) \, dt' \\ 
& - \int_0^t \big( R_1(t') + \e^{2 \a} e^{- 4i t'/\e} (\Id - \Pi_0) (\bar w_1^{2}(t') R_1(t') \big) \, dt'. \end{aligned}\ee
Notice that there is no room for a further integration by parts as there are no more fast
oscillations in time in the leading term in \eqref{w1:implicit:2}.

\subsection{Short-time existence} \label{sec:short-time}

This step relies on a standard fixed point argument, as in Theorem \ref{th:1}. We only sketch the argument and the bounds here. As in Section \ref{sec:end:1}, we use the balls \eqref{def:balls} and define a map $F$ \eqref{def:F}. The first component $F_1$ of $F$ is
$$ %
F_1(w_1,w_2)(t) := w_1^0 - i \e^{2 \a - 1} \int_0^t (\Id - \Pi_0)( |w_1|^{2} w_1)(t') \, dt'- \int_0^t R_1(t') \, dt', \qquad \mbox{if $p = p_1;$}
$$ %
and
$$ %
 \begin{aligned}
 F_1(w_1,w_2)(t) & :=w_1^0 + \frac{\e^{2 \a}}{4} \Big( e^{ - 4 i  t/\e}  (\Id - \Pi_0) \bar w_1(t)^{3} - (\Id - \Pi_0) (\bar w_1^0)^{3}\Big)  \\ & - i \e^{4 \a - 1} \int_0^t (\Id - \Pi_0) \big( \bar w_1(t')^{2} (\Id - \Pi_0) w_1(t')^{3} \big) \, dt' \\ 
& - \int_0^t \big( R_1(t') + \e^{2 \a} e^{- 4 i t'/\e} (\Id - \Pi_0) (\bar w_1(t')^{2} R_1(t') \big) \, dt',  \qquad \mbox{if $p = p_2,$}
\end{aligned}
$$ %
in accordance with \eqref{w1:implicit:2}. 
 The second component of $F$ is defined by  
$$ %
 \begin{aligned}
 F_2(w_1,w_2)(t)  & := e^{- i t \d_x^2/\e^3} w_2^0 - \int_0^t e^{ - i (t - t') \d_x^2/\e^3} \big( f_2(t',w_1(t'), w_2(t')) +  R_2(t')\big) \, dt',
\end{aligned}
$$ %
with notation
$$ 
f_2(t,w_1, w_2) := - \frac{1}{\e^2} e^{i t/\e} q(\e^\a \tilde u_1) \big( e^{it /\e} \d_x w_1 \big) - \frac{i}{\e} e^{it /\e}  q(\e^\a \tilde u_1) \big( e^{- i t /\e} w_2 \big),
$$
with $q = q_1$ or $q = q_2$ \eqref{def:q}. 

In the case $p = p_1,$ bounds very similar to the ones from Section \ref{sec:end:1} yield the estimate 
\be \label{bd:F1:1:th2} \| F_1 \|_{H^1} \leq \| w_1^0\|_{H^1} + 3 \e^{2 \a - 1} c^{2} \rho^{3} T + C_R \rho T,\ee
where $c > 0$ is a norm for the Sobolev embedding $H^1(\T) \hookrightarrow L^\infty(\T),$ and $C_R > 0$ is the constant that appears in the remainder bound. 
If $p = p_2,$ we have similarly 
\be \label{bd:F1:2:th2} \begin{aligned} \| F_1 \|_{H^1}  \leq \| w_1^0\|_{H^1} + \e^{2 \a} \| w_1^0\|_{L^\infty}^{2} \| w_1^0 \|_{H^1} & + \e^{2 \a} c^{2} \rho^{3} + 5 \e^{4 \a - 1} c^{4} \rho^{5} T  + C_R \rho T \big(1 + \e^{2 \a} c \rho^{2}\big).\end{aligned}\ee 
In both cases, we have for $\e < c$ the bound 
\be \label{bd:F2:th2} \| F_2 \|_{L^2} \leq \| w_2^0 \|_{L^2} + \e^{-2 + 2 \a} C_0 \rho^{3} T + C_R \rho T,
\ee
where $C_0 > 0$ depends only on the embedding constant $c.$ Thus for
\be \label{def:Te:th2}
 T_\e := \frac{\e^{2(1 - \a)}}{2 C_0 \rho^{2}},
 \ee
 where $C_0$ is the constant which appears in \eqref{bd:F2:th2}, then for $\e$ small enough (depending on $\rho$), the map $F$ maps the product of balls $B_1(\rho, T_\e) \times B_0(\rho, T_\e)$ into itself. The contraction bounds are similar to the ones of Section \ref{sec:contraction}, and this gives existence and uniqueness in $H^1 \times L^2$ over the small time interval $[0, T_\e].$  
 
 {\it In the rest of the proof of Theorem {\rm \ref{th:2}}, we use the conservation of energy in order to extend the solutions from $[0, T_\e]$ up to time intervals of length $O(1)$ with respect to $\e.$}

\subsection{An improved bound for $w_2$ with the conservation of energy} \label{sec:w2} We have $w_2(t) = F_2(t, w_1(t), w_2(t))$ so that $w_2(t)$ is bounded as in \eqref{bd:F2:th2}. The key is that by the conservation of energy we have a much better bound for $w_2,$ namely 
\be \label{improved:w2:1} \| w_2(t)\|_{L^2} \leq \frac{1 + \e}{1 - \e} \| w_1(t) \|_{L^2}, \qquad \mbox{for all $t \in [0, T_\e],$ if $p = p_1.$}\ee
which comes from \eqref{w2leqw1}. In the case $p = p_2,$ from \eqref{w2leqw1:2} and the fact that $(w_1, w_2) \in B_1(\rho, T_\e) \times B_0(\rho, T_\e)$ with $\rho < 1,$ we deduce
\be \label{improved:w2:2} 
\| w_2 (t) \|_{L^2} \leq (1 + \e^{2\a} C) \| w_1 (t) \|_{L^2}, \qquad \mbox{for all $t \in [0, T_\e],$ if $p = p_2,$}\ee 
for some $C > 0$ which depends only on the embedding constant $c.$ 

\subsection{A bound for $w_1$}  \label{sec:w1} 
For $w_1$ in the case $p = p_1,$ we use the bound \eqref{bd:F1:1:th2}, which implies
\be \label{bd:w1:1}
 \| w_1(t) \|_{H^1} \leq \frac{\rho}{6} + (1 + \e^{2 \a - 1}) C \rho t,\ee
and if $p = p_2,$ we use \eqref{bd:F1:2:th2}, which implies
\be \label{bd:w1:2} \begin{aligned}
\| w_1(t) \|_{H^1}  \leq \frac{\rho}{6} + \e^{2 \a} C \rho^3 + (1 + \e^{4 \a - 1}) C \rho t.\end{aligned}\ee 

\subsection{Continuation: a first step} In the case $p = p_1,$ we let $\a = 1/2,$ and in the case $p = p_2,$ we let $\a = 1/4.$ Then, by \eqref{improved:w2:1}-\eqref{improved:w2:2} and \eqref{bd:w1:1}-\eqref{bd:w1:2}, we obtain for $\e$ small enough, depending on $\rho,$ the bound 
$$ \begin{aligned} \max( \| w_1(T_\e) \|_{H^1}, \| w_2(T_\e)\|_{L^2}) & \leq \frac{\rho}{5} + 2 C \rho T_\e =: \frac{\rho'}{6}.
\end{aligned}$$ 
Associated with the new ``initial" radius $\rho',$ by the arguments of Section \ref{sec:short-time} we have an existence time $T'_\e,$ defined in terms of $\rho'$ just like $T_\e$ was defined in terms of $\rho:$ 
$$ T'_\e := \frac{\e^{2(1 - \a)}}{2 C_0 (\rho')^{2}},$$
with the same constant $C_0$ as in \eqref{def:Te:th2}. 
At this point we extended the short-time solution of Section \ref{sec:short-time} from $[0, T_\e]$ up to $[0, T_\e + T'_\e].$

\subsection{A sequence of continuation steps} Consider the size of the solution at $T_\e + T'_\e.$ Again, the bound for $w_2$ derives from the bound for $w_1,$ by the arguments of Section \ref{sec:w2}. For $w_1$ we can use the bounds of Section \ref{sec:w1}, which are valid so long as $\| w_1 \|_{H^1} \leq \rho$ and $\| w_2 \|_{L^2} \leq \rho.$ Thus 
$$ \max( \| w_1(T_\e + T'_\e)\|_{H^1}, \| w_2(T_\e + T'_\e)\|_{L^2}) \leq \frac{\rho}{5} + 2 C \rho (T_\e + T'_\e).$$
This gives a new ``initial" radius $\rho_{2,\e}.$ An associated existence time $T_{2,\e}$ is given by the arguments of Section \ref{sec:short-time}. 

Thus we have radii $\rho_{j,\e},$ with $\rho_{0,\e} = \rho,$ $\rho_{1,\e} = \rho',$ and associated existence times $T_{j,\e},$ with $T_{0,\e} = T_\e,$ $T_{1,\e} = T'_\e,$ defined for $j \geq 1$ by  
$$ %
 \rho_{j,\e} = \frac{6}{5} \rho + 12 C \rho \cdot \sum_{0 \leq k \leq j-1} T_{k,\e}, \qquad T_{j,\e} := \frac{\e^{2(1 - \a)}}{2 C_0 \rho_{j,\e}^{2}},
 $$ %
and we can keep going so long as $w_1$ is bounded by $\rho$ in $H^1$ norm and $w_2$ is bounded by $\rho$ in $L^2$ norm.

\subsection{Existence up to time $O(1)$} 

The (a priori finite) sequence $(\rho_{j,\e})_j$ is increasing, with $\rho_{0,\e} = \rho.$ Let $j$ be such that the sequence $(\rho_{k,\e})_k$ is defined up to index $j,$ and  $\rho_{k,\e} \leq 2 \rho$ for all $k \leq j.$ Such an index $j$ exists if $\e$ is small enough, depending on $\rho:$ we can take $j = 1.$ We will show by induction that $j$ can be chosen to be very large, depending on $\e.$ 

For any such $j,$ we have 
\be \label{low:t} \frac{\e^{2(1 - \a)}}{2 C_0 (2 \rho)^{2}} \leq T_{k,\e} \leq \frac{\e^{2(1 -  \a)}}{2 C_0 \rho^{2}}, \qquad \mbox{for all $k \leq j,$}\ee
implying
$$ \rho_{k,\e} \leq \frac{6}{5} \rho + \frac{6 C}{C_0 \rho} k \e^{ 2(1-\a)}, \qquad \mbox{for all $k \leq j.$}$$
From \eqref{bd:w1:1}-\eqref{bd:w1:2}, we deduce 
$$ %
\max\Big( \Big\| w_1\big(\sum_{0 \leq  k \leq j} T_{k,\e}\Big) \Big\|_{H^1}, \Big\| w_2(\sum_{0 \leq k \leq j} T_{k,\e}) \Big\|_{L^2}\Big) \leq \frac{\rho}{5} + \frac{C}{C_0 \rho} j \e^{2(1 - \a)}.
$$ %
In particular, by an immediate induction, we see that
\be \label{def:j} j := \left[ \frac{C_0 \rho^{2}}{2 C \e^{2(1 - \a)}} \right] - 1\ee
is such that the conditions $\rho_{k,\e} \leq 2 \rho$ are satisfied for all $k \leq j.$ 
(In \eqref{def:j}, $[ \, \cdot \, ]$ denotes the integer part.) Thus for this value of $j$ \eqref{def:j}, we have \eqref{low:t}, which means that we can extend the solution up to time
$$ t_\star := \sum_{0 \leq k \leq j} T_\e  \geq j  \frac{\e^{2(1 - \a)}}{2 C_0 (2 \rho)^{2}} \geq \frac{1}{16 C},$$
for $\e$ small enough.  		 
 Going back up the chain of changes of variables, the solution $(w_1, w_2)$ to \eqref{eq:w1:1}-\eqref{eq:w1:2}-\eqref{eq:w2} gives a solution $u$ to \eqref{1}, over the same time interval and with the same regularity, and the proof is complete.

\section{Proof of Theorem \ref{th:3}} 

The main difference with the proof of Theorem \ref{th:2} is that we perform here a large number $O(|\ln \e|)$ of integration by parts in time in the ordinary differential equation for $w_1$ that is derived from the normal form reduction.

\subsection{The initial-value problem for high-frequency data} 

System \eqref{modified} takes the form
$$ %
\d_t u + {\mathcal C} \big( A \d_x u + q_0(u_1) B \d_x u\big)  + \e i D \d_x^2 u = 0,
$$ %
with $A,$ $B$ and $D$ as in Section \ref{sec:proof:starts}, and $q_0(u_1) = 3 u_1^2.$ The operator ${\mathcal C}$ is the component-wise complex conjugation in $\C^2:$ 
\be \label{def:mathcalC}
 {\mathcal C}: \quad z = (z_1,z_2) \in \C^2 \longrightarrow (\bar z_1, \bar z_2) \in \C^2,\ee
 where $\bar z$ is the complex conjugate of $z \in \C.$ We consider concentrating data with amplitude $\l > 0,$ for some $\l$ which will be chosen small, but independent of $\e:$ 
\be \label{2:modified}
u(0,x) = \l u^0(x/\e^2),
\ee
where $u^0$ is independent of $\e$ and belongs to $H^1(\T) \times L^2(\T).$
We look for $u$ in the form 
$$ %
u(t,x) = \l \tilde u(t,x/\e^2) = (\l \tilde u_1, \l \tilde u_2)(t,x/\e^2).
$$ %
Thus the initial-value problem \eqref{modified}-\eqref{2:modified} takes the form 
\be \label{4:modified} \left\{ \begin{aligned} 
\d_t \tilde u + \frac{1}{\e^2} {\mathcal C} \big(A  \d_x \tilde u +  q_0(\l \tilde u_1) B \d_x \tilde u \big) + \frac{i}{\e^3} D \d_x^2 \tilde u & = 0, \qquad t \geq 0, \\ \tilde u(0,x) & =  u^0(x),
\end{aligned}\right.\ee
with $x \in \T.$ Just like in the regularized Euler-Van-der-Waals system \eqref{vdw:i}, the mean mode of the solutions to \eqref{4:modified} is conserved over time. In particular, given data $u^0$ such that $\Pi_0 u^0 = 0,$ we have 
 $\Pi_0 \tilde u(t) \equiv 0,$
so long as $\tilde u$ is defined. 

\subsection{The conservation law in the rescaled spatial frame} 

For the putative solution $(\tilde u_1,\tilde u_2)$ of \eqref{4:modified}, the energy is 
$$ %
{\mathcal E}(t) := \frac{\e^{2} \l^4}{4} \int_\T \Re e \, u_1(t,x)^4  \, dx  -  \frac{\e^{2} \l^2}{2} \int_{\T} \Re e \,  \tilde u_1(t,x)^2 \, dx  + \frac{\e^{2} \l^2 }{2} \| \tilde u_2(t,\cdot) \|_{L^2}^2 \equiv {\mathcal E}(0) \leq 0,
$$ %
by assumption \eqref{negative:energy:modified}. 
This implies 
\be \label{v2leqv1:p2:modified}
 \| \tilde u_2(t) \|_{L^2}^2 \leq \int_{\T} \Re e \,  \tilde u_1(t,x)^2 \, dx - \frac{\l^2}{2} \int_\T \Re e \, \tilde u_1(t,x)^4  \, dx.
\ee

\subsection{Change of variable to normal form} \label{sec:normal:form:modified} 

We now let
$$ %
 v = (\Id + \e M)^{-1} \tilde u, \qquad \mbox{with $M  = - ( m \circ {\mathcal C}) J,$ \quad $\dsp{J := \left(\begin{array}{cc} 0 & 1 \\ 1 & 0 \end{array}\right),}$}
$$ %
where ${\mathcal C}$ is the component-wise complex conjugation in $\C^2$ \eqref{def:mathcalC}. In view of Lemma \ref{lem:m}, we have the bounds, for $0 < \e < 1:$ 
$$ %
(1 - \e) \| v \|_{L^2} \leq  \| \tilde u \|_{L^2} \leq (1 + \e) \| v \|_{L^2},
$$ %
and 
$$ %
(1 -\e C ) \| v \|_{L^\infty} \leq \| \tilde u \|_{L^\infty} \leq (1 + \e C ) \| v \|_{L^\infty},
$$ %
with the same positive constant $C$ as in the pointwise estimate \eqref{est:m:pointwise} in Lemma \ref{lem:m}. Estimate \eqref{bd:M:L2} extends to Sobolev spaces, and we have 
$$ %
(1 - \e) \| v \|_{H^{s'}} \leq  \| \tilde u \|_{H^{s'}} \leq (1 + \e) \| v \|_{H^{s'}}, \qquad \mbox{for any $s'\in \R.$} 
$$ %
We denote $(v_1, v_2)$ the components of $v.$  

\subsection{Conservation of the mean mode and energy} \label{sec:cons:check:modified}

By conservation of the mean mode \eqref{conserved}, assumption on the datum, and $\Pi_0 \circ m\equiv 0,$ we have  
 $\Pi_0 v_1  \equiv0$ and  $\Pi_0 v_2 \equiv0,$
so long as $v_j$ are defined. %
Using \eqref{est:m:pointwise} we find, starting from \eqref{v2leqv1:p2:modified}, the bound \eqref{w2leqw1:2}, just like in the proof of Theorem \ref{th:2}.  

\subsection{Cancellation and the reduced system} Details of the computation are given in Section \ref{sec:normal:form}. The key cancellation here takes the form
$$  [i D \d_x^2, M] + {\mathcal C} A \d_x = - [J,D] \d_x \circ {\mathcal C} +  {\mathcal C} A \d_x = 0,
$$ %
since ${\mathcal C}$ commutes with $A$ (because $A$ has real entries), and with $\d_x.$  
The reduced system is 
\be \label{reduced:system:modified}
 \d_t v + \frac{i}{\e^3} D \d_x^2 v  +  \frac{1}{\e^2} {\mathcal C} \big( q_0(\l \tilde u_1) B \d_x v \big) + \frac{1}{\e} E  v + R v = 0,
 \ee
 where 
$$ %
E v := \left(\begin{array}{c} i v_1 + i \l^2 (\Id - \Pi_0) (v_1^3) \\ i (1 - \overline{q_0(\l \tilde u_1)})  v_2 \end{array}\right).
$$ %
 and the remainder $R$ is defined in terms of $E$ just like in Section \ref{sec:normal:form}:  
$$ %
 R v := - (\Id + \e M)^{-1} M  E v + \left(\begin{array}{c} \e^{-1}\big(  m (q_0(\l \tilde u_1) \d_x v_1) -  m( q_0(\l v_1) \d_x v_1) \big) \\ 0 \end{array}\right).
 $$ %
Just like in Section \ref{sec:normal:form}, the system \eqref{reduced:system:modified} is partially decoupled. The equation in $v_1$ is a non-linear ordinary differential equation: 
$$ %
 \d_t v_1 + \frac{i}{\e} v_1 + \frac{i \l^2}{\e}  (\Id - \Pi_0) (v_1)^3 = - (R v)_1 ,
$$ %
and the equation in $v_2$ is a semi-linear Schr\"odinger equation: 
$$ %
\begin{aligned} 
 \d_t v_2 + \frac{i}{\e^3} \d_x^2 v_2 + \frac{i}{\e} (1 - \overline{q_0(\l \tilde u_1)}) v_2 & = - \frac{1}{\e^2} \overline{q(\l \tilde u_1) \d_x v_1} - (R v)_2.
 \end{aligned}
$$ %
In \eqref{eq:checku1}-\eqref{eq:checku2} we took into account the fact that $\Pi_0 v_j \equiv0,$ as described in Section \ref{sec:cons:check}. We let
$$ %
  (w_1, w_2)(t,x) := e^{- i t/\e} (v_1, v_2)(t,x),
$$ %
so that the system in $(w_1,w_2)$ is
\be \label{w1:w2} \left\{\begin{aligned}
 \d_t w_1 & = - \frac{i \l^2}{\e} e^{2 i t/\e} (\Id - \Pi_0) (w_1^3) + R_1, \\ 
 \d_t w_2 + \frac{i}{\e^3} \d_x^2 w_2 & = \frac{3 \l^2}{\e^2} e^{- 2 i t/\e} \overline{(\tilde u_1)^2} \d_x \bar w_1 + \frac{3 i \l^2}{\e} \overline{(\tilde u_1)^2} w_2 +  R_2, 
 \end{aligned}\right. 
\ee
with $R_j := -  e^{- i t /\e} (R v)_j.$ 

\subsection{A sequence of integrations by parts} \label{sec:ibp:modified}
We integrate in time equation \eqref{w1:w2}(i) in $w_1$ and find the implicit representation
 \be \label{implicit:w1}
 w_1(t) =  w_1^0 - i \e^{-1} \l^2 \int_0^t e^{2 i t'/\e} (\Id - \Pi_0) w_1(t')^3  \, dt' + \int_0^t  R_1(t')  \, dt'.
  \ee 
We can integrate by parts in time in the first integral above. By doing so, we gain a factor $\l^2.$ This is a process that can be repeated an arbitrarily large number of times, as opposed to the previous case, see subsection \ref{IBP}. After $n-1$ integration by parts, the $O(\e^{-1})$ term in the implicit representation has a prefactor $\l^{2 n}.$ When $n$ is so large that $\e^{-1} \l^{2n} = O(1)$ in the limit $\e \to 0,$ that is for $n = O(|\ln \e|),$ then all the terms in the implicit representation of $w_1$ are bounded in time $O(1),$ and this is how we can reach an existence time $O(1)$ starting from data with small but $\e$-independent amplitude. 

In order to formalize the argument, we let
\be \label{def:mu:f}
\mu := - i \l^2 e^{2 i t/\e} , \qquad f: u \in L^\infty \longrightarrow f(u) = (\Id - \Pi_0) u^3 \in L^\infty.
\ee
Then, the ordinary differential equation in $w_1$ is 
$$ %
\e \d_t w_1 = \mu f(w_1) + \e R_1,
$$ %
and the implicit representation \eqref{implicit:w1} takes the form
$$ %
 w_1(t) =  w_1^0 +\frac{1}{\e} \int_0^t \mu f(w_1(t'))  \, dt' + \int_0^t  R_1(t')  \, dt'.
$$ %
We integrate by parts to find
\be \label{implicit:1} \begin{aligned}
w_1(t) & = w_1^0 + \frac{1}{2 i} (\mu f(w_1(t)) - \mu(0) f(w_1^0)) - \frac{1}{2i} \int_0^t \mu \d_t f(w_1(t')) \, dt' + \int_0^t R_1(t') \, dt' \\ & = w_1^0 + \frac{1}{2 i} (\mu f(w_1(t)) - \mu(0) f(w_1^0)) - \frac{1}{2i \e} \int_0^t \mu^2 f'(w_1(t')) f(w_1(t')) \, dt' \\ & + \int_0^t (1 - \frac{\mu}{2 i}  f'(w_1(t'))) R_1(t') \, dt'.
\end{aligned} 
\ee
Above the leading $O(\e^{-1})$ term now has a $\mu^2$ factor: this is the gain of a power of $\l^2$ that we announced. We denote
$$ f_0(u) := f(u), \qquad f_1(u) := f'(u) \cdot f(u),$$ and define by induction
$$ f_{n+1}(u) := f'_n(u) \cdot f(u), \quad \mbox{for $n \geq 1,$}$$
so that
$$ \d_t f_{n}(w_1) := f'_n(w_1) \d_t w_1 = \frac{\mu}{\e} f_{n+1}(w_1) +  f'_n(w_1) R_1.$$ 
With this notation, the implicit representation \eqref{implicit:1} is 
$$ %
\begin{aligned} 
w_1(t) = w_1^0 + \frac{1}{2 i} (\mu f(w_1(t)) - \mu(0) f(w_1^0)) & - \frac{1}{2i \e} \int_0^t \mu^2 f_1(w_1(t')) \, dt' \\ &  + \int_0^t (1 - \frac{\mu}{2 i} f'_0(w_1(t'))) R_1(t') \, dt',
\end{aligned} 
$$ %
and, integrating by parts again, we find
$$ %
\begin{aligned}
w_1(t) & = w_1^0 + \frac{1}{2 i} (\mu f_0(w_1(t)) - \mu(0) f_0(w_1^0)) - \frac{1}{2 i} \frac{1}{4i} (\mu^2 f_1(w_1(t)) - \mu^2(0) f_1(w_1^0)) \\ & + \frac{1}{(2i)(4 i) \e} \int_0^t \mu^3 f_2(w_1(t')) \, dt' \\ & + \int_0^t (1 - \frac{\mu}{2 i} f'_0(w_1(t')) + \frac{\mu^2}{(2i)(4i)} f'_1(w_1(t'))) R_1(t') \, dt'.
\end{aligned} 
$$ %
We deduce that for all $n \geq 1,$ we have
\be \label{implicit:n} 
w_1(t) = P_n(w_1, w_1^0) + \frac{(-1)^n}{\e \Pi_{j=1}^n (2ij)} \int_0^t \mu^{n+1} f_n(w_1(t'))\,dt' + \int_0^t {\bf R}_n(t') \,dt',
\ee
with notation
\be \label{def:P_n}
P_n(w_1, w_1^0) := w_1^0 + \sum_{0 \leq k \leq n-1} \frac{(-1)^k}{\Pi_{j=1}^{k+1}(2i j)} \big( \mu^{k+1} f_k(w_1(t)) - \mu^{k+1}(0) f_k(w_1^0)\big), 
\ee
and
\be \label{def:Rn}
{\bf R}_n := \Big(1 + \sum_{0 \leq k \leq n-1} \frac{(-\mu)^{k+1}}{\Pi_{j=1}^{k+1}(2i j)} f'_{k}(w_1) \Big) R_1 .
\ee

\subsection{Short-time existence} Let $(w_1^0, w_2^0)$ be given in $H^1 \times L^2,$ satisfying \eqref{negative:energy:modified}. We define the map $F$ on $B_1(\rho, T) \times B_0(\rho, T)$ (these balls are defined in \eqref{def:balls}) by $F = (F_1,F_2)$ with%
$$ %
F_1(w_1,w_2) := P_n(\mu, w_1, w_1^0) + \frac{(-1)^n}{\e \Pi_{k=1}^{n} (2 i k)} \int_0^t \mu^{n+1} f_n(w_1(t'))\,dt' + \int_0^t {\bf R}_n(t') \,dt',
$$ %
and
$$ %
F_2(w_1,w_2) := e^{- i t \d_x^2/\e^3} w_2^0 + \int_0^t e^{- i (t - t') \d_x^2/\e^3} \Big(\frac{3 \l^2}{\e^2} e^{- 2 i t/\e} \overline{(\tilde u_1)^2} \d_x \bar w_1 +  \frac{3 i \l^2}{\e} \overline{(\tilde u_1)^2} w_2 +  R_2\Big) (t') \, dt'. 
$$ %

In view of the system \eqref{w1:w2} satisfied by $(w_1,w_2),$ and the implicit representation \eqref{implicit:n}, a fixed point of $F$ is a solution to \eqref{w1:w2} in $C^0([0, T], H^1 \times L^2).$ 
By a straightforward induction, we find the bounds
\be \label{bd:fn}
\| f_n(u) \|_{H^1} \leq \prod_{0 \leq k \leq n} (2 k + 1) \cdot \| u \|_{L^\infty}^{2 (n+1)} \| u \|_{H^1}, \qquad \mbox{for $n \geq 1,$ for all $u \in H^1(\T),$}
\ee
and, for all $u, v \in H^1$ and $n \geq 1:$  
\be \label{bd:f'n}
\| f'_n(u) v \|_{H^1} \leq \prod_{0 \leq k \leq n} (2k + 1) \Big( \| u \|_{L^\infty}^{2 (n+1)} \| v \|_{H^1} + \| u \|_{L^\infty}^{2 n+1} \| v \|_{L^\infty} \| u \|_{H^1}\Big).
\ee
With the definition of $\mu$ in \eqref{def:mu:f}, this implies that $P_n$ defined in \eqref{def:P_n} satisfies the bound
$$ \| P_n(\mu, w_1, w_1^0) \|_{H^1} \leq \| w_1^0\|_{H^1} + \sum_{0 \leq k \leq n-1}  (c \l)^{2(k+1)} \big( \| w_1 \|_{L^\infty}^{2(k+1)} \| w_1 \|_{H^1} + \| w_1^0 \|_{L^\infty}^{2(k+1)} \| w_1^0 \|_{H^1} \big).
$$
With $\| w_1 \|_{H^1} \leq 1$ and $\| w_1^0\|_{H^1} \leq 1,$ this implies
\be \label{bd:Pn}
 \| P_n(\mu, w_1, w_1^0) \|_{H^1} \leq \Big(1 + \frac{(c \l)^2}{(1 - (c \l)^2)^2}\Big) \| w_1^0\|_{H^1} + \frac{(c \l)^2}{(1 - (c \l)^2)^2} \| w_1 \|_{H^1},
\ee 
where $c > 0$ is a constant for the Sobolev embedding $H^1(\T) \hookrightarrow L^\infty(\T).$ 
Besides, by \eqref{bd:fn}, for $w_1 \in B_1(\rho, T)$ with $\rho < 1$ we have 
\be 
\label{bd:Qn}
\Big\| \frac{1}{\e \Pi_{k=1}^n (2ik)} \int_0^t \mu^{n+1} f_n(w_1(t'))\,dt' \Big\|_{H^1} \leq \frac{2 n + 1}{\e}(c \l)^{2(n + 1)} \rho T. 
\ee
Given $\l$ such that $c \l < 1,$ we now choose $n$ such that 
\be \label{cond:n} 
\frac{2 n + 1}{\e}  (c \l)^{2(n + 1)} \leq 1.
\ee
This means in particular
$$ n > \frac{|\ln \e|}{|\ln (c \l)|}.$$
The last term in $F_1$ involves the remainder ${\bf R}_n$ described in \eqref{def:Rn}. With \eqref{bd:f'n}, we have the bound
$$ \| {\bf R}_n \|_{H^1} \leq \Big(1 + (1 + \rho) \sum_{0 \leq k \leq n-1}  (c \l)^{2(k+1)}\Big) \| R_1 \|_{H^1}.$$
The remainder $R_1$ is entirely analogous to the $R_1$ we had in the proof of Theorem \ref{th:1}, so that $\| R_1(t) \|_{H^1} \leq C_R \rho$ for $t \in [0,T],$ just like in Section \ref{sec:remainder}. Thus
\be \label{bd:Rn}
\Big\| \int_0^t {\bf R}_n(t') \, dt' \Big\|_{H^1} \leq \Big(1 + \frac{(1 + \rho) (c \l)^2}{(1 - (c \l)^2)^2} \Big) C_R \rho T.
\ee 
Putting \eqref{bd:Pn}, \eqref{bd:Qn} and \eqref{bd:Rn} together, we find, for $\dsp{\l \leq \frac{1}{c \sqrt 2},}$ for $n$ satisfying \eqref{cond:n}, and for $(w_1, w_2) \in B_1(\rho, T) \times B_0(\rho, T)$ the bound 
\be \label{bd:F1:modified}
\| F_1(w_1,w_2)\|_{H^1} \leq 3 \| w_1^0 \|_{H^1} + 2 (c \l)^2 \rho + \rho T + (3 + 2\rho) C_R \rho T.
\ee 
The bound for $F_2$ is derived as in the proof of Theorem \ref{th:1}. We find
\be \label{bd:F2:modified} 
\| F_2 \|_{L^2} \leq \| w_2^0 \|_{L^2} + \e^{-2} C \rho^{3} T + C_R \rho T,
\ee
for some constant $C > 0$ which depends only on $c.$ From \eqref{bd:F1:modified}-\eqref{bd:F2:modified}, and contraction bounds entirely analogous to the ones from Section \ref{sec:contraction}, we derive the existence and uniqueness of a fixed point to $F$ in $B_1(\rho, T_\e) \times B_0(\rho, T_\e),$ with $T_\e = O(\e^2),$ for $\rho < 1.$ 
 
 \subsection{Continuation} The arguments are identical to the continuation arguments at the end of the proof of Theorem \ref{th:2}. The second component $w_2$ of the unknown is controlled by the first component $w_1,$ thanks to the assumption of a negative initial energy and the conservation of energy: indeed, see \eqref{v2leqv1:p2:modified}, which, just like in Section \ref{sec:cons:check}, implies the bound \eqref{w2leqw1:2}. As a consequence, the singular bound in $F_2$ \eqref{bd:F2:modified} is used only to initiate the fixed point argument. As we saw above, this singular bound yields a small $O(\e^2)$ existence time. But then the continuation argument uses only the $F_1$ bound \eqref{bd:F1:modified}, which is not singular in $\e,$ hence an eventual existence time which is $O(1).$  

 \subsection*{Acknowledgment.} The first named author acknowledges support from the University of Padova STARS project ``Linear and Nonlinear Problems for the Dirac Equation" (LANPDE).

\end{document}